\documentclass[letterpaper,11pt]{amsart}

\usepackage[margin=1.2in]{geometry}
\usepackage{amsmath,amsthm,amssymb}
\usepackage{xspace,xcolor}
\usepackage[breaklinks,colorlinks,citecolor=teal,linkcolor=teal,urlcolor=teal,pagebackref,hyperindex]{hyperref}
\usepackage[alphabetic]{amsrefs}
\usepackage[all]{xy}
\usepackage{color}

\theoremstyle{plain}
\newtheorem{thm}{Theorem}[section]

\newtheorem{lem}[thm]{Lemma}
\newtheorem{prop}[thm]{Proposition}
\newtheorem{cor}[thm]{Corollary}

\theoremstyle{definition}

\newtheorem{eg}[thm]{Example}

\theoremstyle{remark}
\newtheorem{rmk}[thm]{Remark}

\def\Q{{\mathbf Q}}
\def\R{{\mathbf R}}
\def\C{{\mathbf C}}

\def\A{{\mathbf A}}
\def\P{{\mathbf P}}

\def\cD{\mathcal{D}}
\def\cE{\mathcal{E}}
\def\cF{\mathcal{F}}

\def\cH{\mathcal{H}}

\def\cM{\mathcal{M}}

\def\cO{\mathcal{O}}

\def\.{\cdot}
\def\^{\widehat}

\def\({\left(}
\def\){\right)}

\renewcommand{\and}{ \ \ \text{ and } \ \ }

\usepackage{soul}

\begin{document}

\author{Mircea Musta\c{t}\u{a}}

\address{Department of Mathematics, University of Michigan, 530 Church Street, Ann Arbor, MI 48109, USA}

\email{mmustata@umich.edu}

\author{Sebasti\'{a}n Olano}

\address{Department of Mathematics, University of Michigan, 530 Church Street, Ann Arbor, MI 48109, USA}

\email{olano@umich.edu}

\author{Mihnea Popa}

\address{Department of Mathematics, Harvard University,
1 Oxford Street, 
Cambridge, MA 02138, USA}

\email{mpopa@math.harvard.edu}

\author{Jakub Witaszek}

\address{Department of Mathematics, University of Michigan, 530 Church Street, Ann Arbor, MI 48109, USA}

\email{jakubw@umich.edu}

\thanks{M.M. was partially supported by NSF grants DMS-2001132 and DMS-1952399, M.P. by NSF grant
DMS-2040378, and J.W. by NSF grant DMS-2101897}


\begin{abstract}
We study the Du Bois complex $\underline{\Omega}_Z^\bullet$ of a hypersurface $Z$ in a smooth complex algebraic variety in terms 
its minimal exponent $\widetilde{\alpha}(Z)$. The latter is an invariant of singularities, defined as the negative of the greatest root of the reduced Bernstein-Sato polynomial of $Z$, and refining the log canonical threshold. We show that if $\widetilde{\alpha}(Z)\geq p+1$, then the canonical morphism $\Omega_Z^p\to \underline{\Omega}_Z^p$ is an isomorphism, where 
$\underline{\Omega}_Z^p$ is the $p$-th associated graded piece of the Du Bois complex with respect to the Hodge filtration. On the other hand, if $Z$ is singular and $\widetilde{\alpha}(Z)>p\geq 2$, we obtain non-vanishing results for some of the higher cohomologies of $\underline{\Omega}_Z^{n-p}$.
\end{abstract}

\title{The Du Bois complex of a hypersurface and the minimal exponent}

\maketitle

\section{Introduction}

One of the Hodge theoretic objects of great interest associated to a  variety $Z$ -- by which in this paper we always mean a reduced separated scheme of finite type over $\C$ -- is the Du Bois complex (or filtered 
de Rham complex) $\underline{\Omega}_Z^\bullet$, defined in \cite{DuBois}, and later in a slightly different fashion in 
\cite{GNPP}. This is an object in the derived category of filtered complexes on $Z$; when $Z$ is smooth, it is given by the usual algebraic de Rham complex of $Z$, with its ``stupid" filtration. In general,
the (shifted) associated graded objects
$$\underline{\Omega}_Z^p : = {\rm Gr}^p_F \underline{\Omega}_Z^\bullet [p]$$
are objects in the derived category of coherent sheaves which provide useful generalizations of the bundles of $p$-forms
in the smooth case (for example, they feature in an extension of the Akizuki-Nakano vanishing theorem to singular varieties).
The $0$-th filtered piece $\underline{\Omega}_Z^0$ appears extensively in the literature, as it is related to what
has become a quite important class of singularities; recall that $Z$ is said to have Du Bois singularities if the natural morphism 
$\cO_Z \to  \underline{\Omega}_Z^0$ is a quasi-isomorphism. See for instance \cite{KS1} for a nice 
overview of Du Bois singularities and their role in birational geometry. Besides some formal statements and some special classes of singularities, little is known about $\underline{\Omega}_Z^p$ with $p \ge 1$.

The aim of this paper is to study the behavior of these higher filtered graded pieces of $\underline{\Omega}_Z^\bullet$ when $Z$ is a reduced hypersurface in a smooth irreducible algebraic variety $X$, using methods from the theory of Hodge modules. We give both vanishing and non-vanishing statements about various cohomologies of these complexes, in terms of a singularity invariant derived from the Bernstein-Sato polynomial $b_Z(s)$, namely the \emph{minimal exponent}  $\widetilde{\alpha}(Z)$. This is defined as the negative of the greatest root of the reduced Bernstein-Sato polynomial  $\tilde{b}_Z(s) =  b_Z(s)/ (s+1)$, and has been studied extensively in \cite{Saito_microlocal}, \cite{Saito-MLCT}, \cite{MP1}, \cite{MP2}; see \cite[Section~6]{MP2} for a general discussion.

\smallskip

\noindent
{\bf K\"ahler differentials and the Du Bois complex.} 
From now on we assume that $Z$ is a reduced hypersurface in the smooth, irreducible, $n$-dimensional complex algebraic variety $X$. 

M. Saito has shown that $Z$ has Du Bois singularities if and only if
$\widetilde{\alpha} (Z) \ge 1$, which is equivalent to the pair $(X, Z)$ being log-canonical (he also showed that $Z$ has rational singularities if and only if
$\widetilde{\alpha} (Z) > 1$).
Our main result 
says that a part of the Du Bois complex of $Z$ becomes similarly well-behaved as the minimal exponent gets larger.

%\commentbox{Jakub: out of curiosty, would it be possible to give a proof of this result by using V-filtration as in the spirit of the non-vanishing proof?}

\begin{thm}\label{thm1_intro}
If $p$ is an integer such that $0\leq p \leq \widetilde{\alpha}(Z)-1$, 
then the canonical morphism 
$$\Omega_Z^p \to \underline{\Omega}_Z^p$$
is a quasi-isomorphism.\footnote{As part of the proof we show that 
$\Omega_Z^p$ is reflexive for all such $p$, see Remark~\ref{rmk_reflexive}.}
\end{thm}

For any non-negative integer $p$, the singularities for which $\widetilde{\alpha} (Z) \ge p+1$ are sometimes called \emph{$p$-log canonical}, by analogy with the case $p = 0$. Note that the minimal exponent can be explicitly bounded, and can also be computed for certain singularity types. For example, we have 
$\widetilde{\alpha}(Z) = ({\rm dim} ~Z + 1)/m$ for an ordinary singularity of multiplicity $m\geq 2$, and $\widetilde{\alpha}(Z) = \sum w_i$ for a weighted homogeneous isolated singularity of weights $w_1, \ldots, w_n$; see Section \ref{scn:LHFME} for details.

Theorem \ref{thm1_intro} is in fact a special case of a stronger statement, in which the vanishing of each individual 
$\cH^i(\underline{\Omega}_Z^q)$ with $i>0$ is derived from a suitable lower bound on the codimension 
of the locus in $Z$ where the minimal exponent is $<(p+1)$ (i.e.\ the co-support of the Hodge ideal $I_p (Z)$); see Theorem \ref{stronger-vanishing} for the precise statement. One consequence, see Corollary \ref{cor:singular}, is that if the singular locus of $Z$ has dimension $s$, then for all $p \ge 0$ we have 
$$\cH^i(\underline{\Omega}_Z^p)=0 \,\,\,\,\,\,{\rm for} \,\,\,\,\,\, 0 <  i < \dim Z - s - p - 1.$$
In particular this applies to non-Du Bois singularities as well; see Remark \ref{non-DB}.

\smallskip

\noindent
{\bf Vanishing results.}
In view of the connection between the Du Bois complex and sheaves of forms with log poles on a resolution, established by Steenbrink \cite{Steenbrink}, Theorem~\ref{thm1_intro} implies (in fact is almost equivalent to) a local vanishing result for direct images of such sheaves. 
From now on, we assume that $\mu\colon Y\to X$  is a proper morphism that is an isomorphism over $X\smallsetminus Z$, such that $Y$ is smooth and $E=(\mu^*Z)_{\rm red}$ is a simple normal crossing divisor. 

\begin{cor}\label{cor1_intro}
If $p$ is a nonnegative integer such that $\widetilde{\alpha}(Z)\geq p+1$, then
%\begin{equation}\label{eq_corB}
$$R^i \mu_*\Omega_Y^p({\rm log}\,E)(-E)=0\quad\text{for}\quad i>0.$$
%\end{equation}
\end{cor}
This is an extension of the following fact, due to Steenbrink \cite[Proposition 3.3]{Steenbrink} and to Schwede \cite[Theorem 4.3]{Schwede} in a more general setting: if $Z$ is Du Bois, then the canonical morphism $\cO_Z \to \R \mu_* \cO_E$ is a quasi-isomorphism;\footnote{These two conditions are in fact equivalent, even when $Z$ is not necessarily a hypersurface in a smooth variety.} this translates into the vanishing of $R^i \mu_* \cO_Y (-E)$ for $i >0$.

We also deduce from Theorem \ref{thm1_intro} the following version of global Akizuki-Nakano vanishing for hypersurfaces with high minimal exponent.

\begin{cor}\label{cor0_intro}
If $p$ is a nonnegative integer such that $\widetilde{\alpha}(Z)\geq p+1$, then for every ample line bundle $L$ on $Z$, we have
$$H^q(Z, \Omega^p_Z\otimes L)=0\quad \text{for}\quad q>n-1-p.$$
\end{cor}

Using the same approach as in the proof of 
Theorem \ref{thm1_intro}, 
we also obtain a vanishing result under a slightly weaker assumption on the minimal exponent:

\begin{thm}\label{thm3_intro}
If $q$ is a nonnegative integer such that $q \le \widetilde{\alpha}(Z)$, then
$$\cH^{n-q-1}(\underline{\Omega}_Z^q)=0$$ 
unless $q=n-1$; moreover this last equality can only hold if either $Z$ is smooth or $q =1$ and $Z$ is a nodal curve on a surface. In particular, we have
\begin{equation}\label{eq_thm3_intro}
R^{n-q} \mu_*\Omega_Y^q({\rm log}\,E)(-E)=0.
\end{equation}
\end{thm}

Note that $\underline{\Omega}_Z^q$ for $q = \lfloor \widetilde{\alpha}(Z) \rfloor$ is the first graded piece of the Du Bois complex which is not covered by Theorem \ref{thm1_intro}. Theorem~\ref{thm3_intro} shows that the highest degree cohomology sheaf  of this graded piece that could possibly be non-trivial, does in fact vanish. The fact that this is the 
top possible non-trivial cohomology is a consequence of the general vanishing $\cH^{p}(\underline{\Omega}_Z^q)=0$ for $p \geq n-q$ and every $q$. 
This in turn is related to a theorem of Steenbrink, see \cite[Theorem~2]{Steenbrink}, stating that
$$R^i \mu_*\Omega^j_Y({\rm log}\,E)(-E)=0\quad \text{for}\quad i+j>n$$
(in the case of hypersurfaces this is easy to prove, see Section~\ref{rmk_Steenbrink}, but Steenbrink's result holds more generally when $Z$ has
arbitrary codimension in $X$).

\noindent
\emph{Remark.}
When $q=0$, the vanishing in (\ref{eq_thm3_intro}) is trivial, while for $q=1$ it is a special case of a result of
Greb, Kov\'{a}cs, Kebekus and Peternell, see \cite[Theorem~14.1]{GKKP}, which applies
to general log canonical pairs. It is also interesting to note that a related result, namely
$$R^{n-q} \mu_*\Omega_Y^q({\rm log}\,E) = 0  \,\,\,\,\,\,{\rm for} \,\,\,\,\,\,q \le \lceil \widetilde{\alpha}(Z)\rceil$$
appears in \cite[Corollary~C]{MP3}. Despite the similarity, its proof is of a very different flavor.

\smallskip

\noindent
{\bf Non-vanishing result and applications.}
Changing gears, we also give a non-vanishing result for the cohomology of certain graded pieces of the Du Bois complex 
when the minimal exponent is large. 

\begin{thm}\label{thm4_intro}
Suppose that $Z$ is defined in $X$ by $f\in\cO_X(X)$. If $p\geq 2$ is an integer such that $\widetilde{\alpha}(Z)>p$,
then for every singular point $x\in Z$, the following hold:
\begin{enumerate}
\item[i)] We have an isomorphism $\cH^{p-1}(\underline{\Omega}_Z^{n-p})_x\simeq \cO_{X,x}/\big(J_f+(f)\big)$, where $J_f$
is the Jacobian ideal of $f$.\footnote{In an open subset with algebraic coordinates $x_1,\ldots,x_n$, the ideal $J_f$ is generated by $\partial f/\partial x_1,\ldots,\partial f/\partial x_n$.} In particular, $\cH^{p-1}(\underline{\Omega}_Z^{n-p})_x\neq 0$.
\item[ii)] If $x$ is an isolated singularity of $Z$ and $p\geq 3$, then 
$\cH^{p-2}(\underline{\Omega}_Z^{n-p})_x\simeq (J_f : f)/J_f$. In particular, $\cH^{p-2}(\underline{\Omega}_Z^{n-p})_x\neq 0$
${\rm (}$while $\cH^{i}(\underline{\Omega}_Z^{n-p})_x=0$
for $0< i< p-2$${\rm )}$.
\end{enumerate}
\end{thm}

Regarding the statement, it is worth noting that, as before, $\cH^{p-1}$ is the top possible nonzero cohomology of 
$\underline{\Omega}_Z^{n-p}$; see Section \ref{rmk_Steenbrink}. Though the starting point is similar, the proof is somewhat different from that of the vanishing results, in that it appeals to the $V$-filtration (and its connection with the minimal exponent), as well as to duality for nearby and vanishing cycles.

The non-vanishing result has some interesting consequences. The first stems from the fact that 
if $Y$ is a variety with quotient or toroidal singularities, then $\cH^i(\underline{\Omega}_Y^p)=0$ for all $i\geq 1$ and all $p$; for quotient singularities, see \cite[Section~5]{DuBois}, and for toroidal singularities, see \cite[Chapter~V.4]{GNPP}. Thanks to Theorem~\ref{thm4_intro}, we deduce that in these cases minimal exponents are surprisingly rather small:

\begin{cor}\label{ME-quotient}
If $Z$ is singular and has quotient or toroidal singularities, then $1 < \widetilde{\alpha}(Z)\leq 2$. 
\end{cor}

Note that the upper bound is sharp: the hypersurface defined by 
$x_1x_2-x_3x_4$ in ${\mathbf C}^4$ is toric and its minimal exponent is $2$; see for instance the paragraph after Theorem \ref{thm1_intro}. The lower bound is due to the fact that these are rational singularities.

The second consequence is that the cohomology sheaves $\cH^i (\underline{\Omega}_Z^p)$, with $p, i \ge 1$, are not upper semicontinuous in families.
%do not exhibit upper semicontinuous behavior in families. 
This should be contrasted with a result for $p =0$ (when $Z$ is not necessarily a hypersurface)
due to Kov\'acs and Schwede \cite{KS2}, who have shown that nearby deformations of Du Bois singularities are again Du Bois. 

\begin{eg}
Let $f, g \in \C[X_1, \ldots, X_n]$ with $n \ge 5$, be chosen so that $f$ defines a hypersurface with quotient singularities, with a singular point at $0$
(hence $\widetilde{\alpha}_0(f)\leq 2$ by Corollary \ref{ME-quotient}) while $g$ defines a hypersurface with a singular point at $0$ and such that
$\widetilde{\alpha}_0(g)> 2$. Consider the family 
of hypersurfaces parametrized by $\A^1$, defined by $h_t: = t f + (1-t) g$. For $t=1$ we have a hypersurface with quotient singularities, 
hence $\cH^i(\underline{\Omega}_{Z(h_1)}^p)=0$ for all $i\geq 1$ and all $p$. On the other hand, the minimal exponent is 
lower semicontinuous in families, see \cite[Theorem~E(2)]{MP2}, hence for general $t$ we have $\widetilde{\alpha}_0(h_t)> 2$
(and the hypersurface $Z(h_t)$ has a singular point at $0$). Theorem \ref{thm4_intro} then implies that $\cH^{1}(\underline{\Omega}_{Z(h_t)}^{n-2}) \neq 0$.
\end{eg}

We note that since the first version of this paper was written, further progress has been made on this topic: the converse of Theorem~\ref{thm1_intro} 
was proved in \cite{Saito_et_al}, while an analogue of both implications in the case of local complete intersections was proved in \cite{MP4}.

\noindent
{\bf Outline and acknowledgement.}
The paper is organized as follows: we begin by reviewing in the next section some basic facts about the minimal exponent,
the Hodge filtration on the local cohomology sheaf $\cH^1_Z(\cO_X)$, and the
graded pieces of the Du Bois complex. In particular, we recall the description of these graded pieces
in terms of the de Rham complex of $\cH^1_Z(\cO_X)$. The proofs of Theorems~\ref{thm1_intro} and \ref{thm3_intro} are given in Section~3, while the proof of Theorem~\ref{thm4_intro} is the content of Section~4.

We thank the referees for comments that helped us substantially improve the exposition.

\section{Review of the Du Bois complex, Hodge filtration, and minimal exponent}\label{review}

In this section we review some basic facts about the objects in its title that we will need for the proofs of our main results. 

\subsection{Conventions and notation}\label{section_notation}
By a variety we always mean a reduced, separated scheme of finite type over $\C$, possibly reducible.
As in the Introduction, $X$ stands for a smooth, irreducible, $n$-dimensional variety and $Z$ is a nonempty reduced hypersurface in $X$. 

We consider a proper morphism $\mu\colon Y\to X$ that is an isomorphism over $X\smallsetminus Z$, such that $Y$ is smooth and $E=(\mu^*Z)_{\rm red}$ is a simple normal crossing divisor. Such a morphism exists by Hironaka's theorem and, in fact, can be chosen to be projective (but we do not make this assumption unless explicitly mentioned otherwise).

\subsection{Du Bois complex}\label{section_DuBois}
For an introduction to the \emph{Du Bois complex} (sometimes called the \emph{filtered De Rham complex}) and its basic properties, we refer to \cite[Chapter V.3]{GNPP},  \cite[Chapter~7.3]{PetersSteenbrink}, \cite{Steenbrink}, and to the original paper of Du Bois \cite{DuBois}. A useful list of properties is also collected together in \cite[Theorem~4.2]{KS1}.

Recall that for a variety $W$, this is a filtered complex denoted $(\underline{\Omega}_W^{\bullet},F^{\bullet})$. We will only be interested in its graded pieces, suitably shifted:
$$\underline{\Omega}_W^p:={\rm Gr}_F^p\underline{\Omega}_W^{\bullet}[p].$$
For every $p$ this is an element in the bounded derived category of coherent sheaves on $W$, which can be nonzero only when $0 \le p \le \dim W$; moreover, there is a canonical morphism
$$\Omega_W^p \to\underline{\Omega}_W^p,$$ 
which is an isomorphism if $W$ is smooth. The variety $W$ is said to have \emph{Du Bois singularities} if $\cO_W\to \underline{\Omega}^0_W$ is an isomorphism. 

%\subsection{An exact triangle for the graded pieces of the Du Bois complex}\label{section_DuBois}
Suppose now that $X$, $Z$, and $\mu$ are as in Section~\ref{section_notation}.
A key fact, due to Steenbrink \cite[Proposition~3.3]{Steenbrink}, is that for every $p$, we have an exact triangle in the derived category
\begin{equation}\label{eq_exact_trg}
\R \mu_*\Omega^p_Y({\rm log}\,E)(-E)\to\Omega_X^p\longrightarrow \underline{\Omega}_Z^p\overset{+1}{\longrightarrow}.
\end{equation}
We briefly recall the argument, for the benefit of the reader. Since $\mu$ is an isomorphism
over $X\smallsetminus Z$ and $X$ and $Y$ are smooth, we have an exact triangle
\begin{equation}\label{eq_exact_trg2}
\Omega_X^p\longrightarrow \R \mu_*\Omega_Y^p\oplus \underline{\Omega}_Z^p\longrightarrow \R \mu_*\underline{\Omega}_E^p \overset{+1}{\longrightarrow},
\end{equation}
see \cite[Proposition~4.11]{DuBois}. 
We apply the octahedral axiom for the composition
$$\R \mu_*\Omega_Y^p\overset{\alpha}\longrightarrow  \R \mu_*\Omega_Y^p \oplus \underline{\Omega}_Z^p \overset{\beta}\longrightarrow 
\R \mu_*\underline{\Omega}_E^p,$$
where $\alpha=({\rm id},0)$ and $\beta$ is given by the sum of the obvious morphisms. If $Q={\rm cone}(\beta\circ\alpha)$, then we deduce using (\ref{eq_exact_trg2}) that we have an exact triangle
$$\underline{\Omega}_Z^p\longrightarrow Q\longrightarrow \Omega_X^p[1] \overset{+1}{\longrightarrow}.$$
On the other hand, recall that $\underline{\Omega}_E^p=\Omega_Y^p/\Omega_Y^p({\rm log}\,E)(-E)$, see \cite[Example~7.25]{PetersSteenbrink}, which immediately implies that 
$Q[-1]\simeq \R \mu_*\Omega_Y^p({\rm log}\,E)(-E)$. We thus obtain (\ref{eq_exact_trg}).

\subsection{A consequence of Grothendieck duality}
We note that since $\big(\Omega_Y^p({\rm log}\,E)(-E)\big)^{\vee}\simeq
\Omega^{n-p}_Y({\rm log}\,E)\otimes\omega_Y^{-1}$, it follows from Grothendieck duality that
\begin{equation}\label{eq_remark2}
\R \cH om_{\cO_X}\big(\R \mu_*\Omega_Y^p({\rm log}\,E)(-E),\omega_X\big)\simeq \R \mu_*\Omega_Y^{n-p}({\rm log}\,E).
\end{equation}
Since $\R \cH om_{\cO_X}\big(-,\omega_X\big)$ is a duality, we deduce from (\ref{eq_remark2}) that we also have
\begin{equation}\label{eq2_remark2}
\R \cH om_{\cO_X}\big(\R \mu_*\Omega_Y^{n-p}({\rm log}\,E),\omega_X\big)\simeq \R \mu_*\Omega_Y^p({\rm log}\,E)(-E).
\end{equation}

\begin{rmk}
While we are interested in the case of hypersurfaces, the assertions in this and the previous section hold if $Z$ is an arbitrary
subvariety of a smooth, irreducible, $n$-dimensional variety.
\end{rmk}

\subsection{Filtered $\cD_X$-modules and duality}\label{subsection_DR_duality}
Let $\cD_X$ be the sheaf of differential operators on  $X$. Recall that if $\cM$ is a left $\cD_X$-module on $X$, then the de Rham complex of $\cM$ is the complex ${\rm DR}_X(\cM)$:
$$0\to \cM\to \Omega^1_X\otimes_{\cO_X}\cM\to\cdots\to\Omega_X^n\otimes_{\cO_X}\cM\to 0,$$
placed in cohomological degrees $-n,\ldots,0$, with the differentials defined using the usual de Rham differential and the integrable connection on $\cM$.
 If $(\cM,F)$ is a filtered $\cD_X$-module (so that the filtration is compatible with the order filtration on $\cD_X$), then ${\rm DR}_X(\cM)$
 carries an induced filtration, with $F_p{\rm DR}_X(\cM)$ being the subcomplex:
 $$0\to F_{p}\cM\to \Omega_X^1\otimes_{\cO_X}F_{p+1}\cM\to\cdots\to \Omega_X^n\otimes_{\cO_X}F_{p+n}\cM\to 0.$$
 
 We will be interested in the filtered $\cD_X$-modules associated to certain mixed Hodge modules in the sense of M. Saito's theory, see \cite{Saito-MHP}, \cite{Saito-MHM}. For such filtered $\cD_X$-modules there is a duality functor ${\mathbf D}$,  satisfying the following compatibility with the Grothendieck dual of the de Rham complex:
 \begin{equation}\label{eq_compatibility}
 {\rm Gr}_p^F{\rm DR}_X\big({\mathbf D}(\cM)\big)\simeq \R \cH om_{\cO_X}\big({\rm Gr}_{-p}^F{\rm DR}_X(\cM),\omega_X[n]\big)
 \end{equation} for every $p$, see \cite[2.4.5 and 2.4.11]{Saito-MHP}. (See also \cite[4.2.3]{Saito-MHM} for the fact that the 
 functor ${\mathbf D}$ preserves the category of mixed Hodge modules in the algebraic setting.)
 
 In Section~\ref{last_section} we will also make use of right $\cD_X$-modules.
 Recall that there is a canonical equivalence of categories between left and right $\cD_X$-modules such that if
$\cM^r$ is the right $\cD_X$-module corresponding to the left $\cD_X$-module $\cM$, then we have an isomorphism of $\cO_X$-modules
\begin{equation}\label{eq_right_Dmod}
\cM^r\simeq\omega_X\otimes_{\cO_X}\cM.
\end{equation}
For example, the right $\cD_X$-module corresponding to $\cO_X$ is $\omega_X$ and if $Z$ is a reduced hypersurface of $X$, then
the right $\cD_X$-module corresponding to $\cH^1_Z(\cO_X)$ is $\cH^1_Z(\omega_X)$.

We similarly have an equivalence between filtered left and right $\cD_X$-modules and the standard convention is that if $(\cM^r,F)$ corresponds to $(\cM,F)$, then the isomorphism (\ref{eq_right_Dmod}) identifies $F_{p-n}\cM^r$ with $\omega_X\otimes_{\cO_X}F_p\cM$. 
We also note that the filtered De Rham complex associated to $(\cM^r,F)$ can be taken to be the filtered De Rham complex of $(\cM,F)$;
in particular, we have 
$${\rm Gr}^F_{\bullet}{\rm DR}_X\big(\cH^1_Z(\omega_X)\big)={\rm Gr}^F_{\bullet}{\rm DR}_X\big(\cH^1_Z(\cO_X)\big).$$

\subsection{Localization, Hodge filtration, and minimal exponent}\label{scn:LHFME}
The $\cD_X$-module we are interested in is $\cO_X(*Z)$, the sheaf of rational functions on $X$ with poles along $Z$. This underlies a mixed Hodge 
module, hence in particular carries a Hodge filtration; for a detailed study of this filtration, see \cite{MP1}. It is known that the Hodge filtration is contained in the pole order filtration, i.e. for every $p\geq 0$ we have
$F_p\cO_X(*Z)\subseteq \cO_X\big((p+1)Z\big)$, which leads to the definition of the $p$-th Hodge ideal $I_p (Z)$
by the formula
$$F_p\cO_X(*Z) = \cO_X\big((p+1)Z\big) \otimes I_p (Z).$$
Note also that we have a short exact sequence 
$$0 \longrightarrow \cO_X \longrightarrow \cO_X(*Z) \longrightarrow {\mathcal H}_Z^1(\cO_X) :=\cO_X(*Z)/\cO_X 
\longrightarrow 0$$
of filtered $\cD_X$-modules, where $\cO_X$ underlies the trivial mixed Hodge module ${\mathbf Q}_X^H[n]$ and its filtration satisfies ${\rm Gr}^F_q \cO_X =0$ for all $q\neq 0$, while ${\mathcal H}_Z^1(\cO_X)$ coincides with the first local cohomology sheaf of $\cO_X$ along $Z$, and its Hodge filtration is induced by that on $\cO_X(*Z)$. 

\medskip

We now turn to the \emph{minimal exponent} $\widetilde{\alpha}(Z)$ of $Z$, which was originally defined by Saito in
\cite{Saito_microlocal} as the negative of the greatest root of the reduced Bernstein-Sato polynomial $b_Z (s)/ (s+1)$; it is therefore a refinement of the log canonical threshold ${\rm lct}(Z)$, which satisfies 
$${\rm lct}(Z) = {\rm min}\{\widetilde{\alpha}(Z), 1\}.$$
By convention, we have $\widetilde{\alpha}(Z)=\infty$ if and only if $b_Z(s)=s+1$, which is the case if and only if $Z$ is smooth. 
There is also a local version $\widetilde{\alpha}_x (Z)$ of this invariant around each point $x \in Z$, such that 
$\widetilde{\alpha} (Z) = \underset{x \in Z}{\rm min} ~\widetilde{\alpha}_x (Z)$.
See \cite[Section~6]{MP2} for a general discussion and study of the minimal exponent.

It turns out that the minimal exponent governs the complexity of the Hodge filtration in various ways. For instance, it determines how far the Hodge filtration agrees with the pole order filtration $P_\bullet$ on $\cO_X(*Z)$, defined by 
$$P_k\cO_X(*Z)=\cO_X\big((k+1)Z\big) \,\,\,\,\,\,{\rm  for} \,\,\,\,k\geq 0$$
and $P_k\cO_X(*Z)=0$ for $k<0$. Concretely, for a nonnegative integer $p$, we have
$$\widetilde{\alpha}(Z)\geq p+1 \iff F_k\cO_X(*Z)= P_k\cO_X(*Z) \quad\text{for}\quad k\leq p
\iff I_k (Z) = \cO_X\quad\text{for}\quad  k\leq p,$$
see \cite[Corollary~1]{Saito-MLCT}, and also \cite[Corollary C]{MP2}. Under these equivalent conditions we also say 
that the pair $(X,Z)$ is \emph{$p$-log-canonical}, as the case $p =0$ is precisely the case of log-canonical pairs.
It is this interpretation of the minimal exponent that will be used in this paper. 

We have the following numerical criteria for minimal exponents, which in practice can be used as concrete bounds in the context of the results in the Introduction:
\begin{itemize}
\item $\widetilde{\alpha}(Z)\geq 1 \iff Z$ has du Bois singularities $\iff (X, Z)$ is log-canonical. See 
\cite[Theorem~0.5]{Saito09} for the first equivalence, and \cite[Corollary~6.6]{KS1} for the second.
\item  $\widetilde{\alpha}(Z) > 1 \iff Z$ has rational singularities; see \cite[Theorem~0.4]{Saito93}.
\item  If a point $x \in Z$ has multiplicity $m \ge 2$, while the singular locus of its projectivized tangent cone $\P (C_x Z)$ 
has dimension $r$ (with $r = -1$ if $\P (C_x Z)$  is smooth), then 
\begin{equation}\label{eq_bound_min_exp}
\frac{n-r -1}{m} \le \widetilde{\alpha}_x(Z) \le \frac{n}{m};
\end{equation}
see \cite[Theorem~E]{MP2}. (The inequality $\widetilde{\alpha}_x(Z) \le \frac{n}{2}$ also follows from
\cite[Theorem~0.4]{Saito_microlocal}.) In particular 
$\widetilde{\alpha}_x(Z) = \frac{n}{m}$ if $x$ is an ordinary singular point.
\item If $Z$ has a weighted homogeneous isolated singularity, where the variable $x_i$ has weight $w_i$, then 
$\widetilde{\alpha}(Z) = \sum w_i$; see \cite[4.1.5]{Saito09}.
\item Let $\mu\colon Y \to X$ be a morphism as in Section~\ref{section_notation} such that, in addition, 
the strict transform $\widetilde{Z}$ of $Z$ is smooth (in other words, the strict transforms of the irreducible components of $Z$ are pairwise disjoint). 
Define integers $a_i$ and $b_i$ by the expressions
$$\mu^*Z = \widetilde{Z} + \sum_{i=1}^m a_iF_i\quad\text{and}\quad
K_{Y/X} = \sum_{i=1}^m b_iF_i,$$
where $F_1,\ldots,F_m$ are the prime exceptional divisors, and set 
$$\gamma : = \underset{i=1, \ldots, m}{\rm min} \left\{\frac{b_i +1}{a_i}\right\}.$$
Then we have $\widetilde{\alpha}(Z)\geq \gamma$; see \cite[Corollary~D]{MP2}, cf. also \cite[Corollary~1.5]{DM}.
\item The minimal exponent also provides a bound for the generation level of the Hodge filtration $F_\bullet \cO_X (*Z)$, shown in \cite[Theorem~A]{MP3} to be at most $n- 1 - \lceil \widetilde{\alpha}(Z) \rceil$.
\end{itemize}

It is shown in \cite[Proposition~7.4]{MP3} that if $Z$ is singular and $\widetilde{\alpha}(Z)>p$, for a nonnegative integer $p$, then
the codimension in $Z$ of the singular locus $Z_{\rm sing}$ is at least $2p$. We will need the following variant, that can be proved along the same lines:

\begin{lem}\label{lem_singular_locus}
If  $\widetilde{\alpha}(Z)\geq p+1$ for a nonnegative integer $p$,\footnote{It is worth noting that for \emph{any} effective divisor $D$ on $X$, the hypothesis $\widetilde{\alpha}(D)\geq 1$ automatically implies that $D$ is reduced. Indeed, otherwise its log canonical threshold satisfies ${\rm lct}(D) < 1$, hence $\widetilde{\alpha}(D) = {\rm lct}(D)$.}then
$${\rm codim}_Z(Z_{\rm sing})\geq 2p+1.$$
\end{lem}

\begin{proof}
We may and will assume that $Z$ is affine. If $q={\rm dim}(Z_{\rm sing})$, then after successively cutting $X$ with $q$ general hyperplane sections,
we obtain a smooth closed subvariety $Y$ of $X$, of codimension $q$, such that the divisor $Z\vert_Y$ is singular. Moreover, we have
$\widetilde{\alpha}(Z\vert_Y)\geq\widetilde{\alpha}(Z)\geq p+1$ by \cite[Lemma~7.5]{MP3}. In this case, it follows from (\ref{eq_bound_min_exp}) that
$p+1\leq \frac{n-q}{2}$, hence ${\rm codim}_Z(Z_{\rm sing}) =n-1-q\geq 2p+1$. 
\end{proof}

\subsection{The graded pieces of the Du Bois complex via the de Rham complex of ${\mathcal H}_Z^1(\cO_X)$}
The connection between the graded pieces of the Du Bois complex and the Hodge filtration on $\cO_X(*Z)$ is provided by the following result:

\begin{lem}\label{eq_DuBois}
For every $p$, there is an isomorphism
\begin{equation}\label{eq_eq_DuBois}
\underline{\Omega}_Z^p\simeq \R \cH om_{\cO_X}\big({\rm Gr}^F_{p-n}{\rm DR}_X(\cH^1_Z(\cO_X)),\omega_X\big)[p+1].
\end{equation}
\end{lem}
\begin{proof}
Let $\mu\colon Y\to X$ be a morphism as in Section~\ref{section_notation}  assumed, in addition, to be projective. The explicit filtered resolution of the right 
$\cD_Y$-module $\omega_Y(*E)$
corresponding to $\cO_Y(*E)$ given in \cite[Proposition~3.1]{MP1} implies that we have
$${\rm Gr}^F_{p-n}{\rm DR}_Y\big(\cO_Y(*E)\big)\simeq \Omega_Y^{n-p}({\rm log}\,E)[p];$$
cf.\ \cite[Theorem~6.1]{MP1}.
Since $\cO_X(*Z)$ is the push-forward of $\cO_Y(*E)$ (in the category of mixed Hodge modules) and $\mu$ is projective, we obtain using 
Saito's Strictness Theorem, see \cite[Section~2.3.7]{Saito-MHP}, cf.\ \cite[Section C.4]{MP1}, that
$${\rm Gr}^F_{p-n}{\rm DR}_X\big(\cO_X(*Z)\big)\simeq \R \mu_*\Omega_Y^{n-p}({\rm log}\,E)[p].$$
Moreover, the canonical morphism 
$$\alpha\colon {\rm Gr}^F_{p-n}{\rm DR}_X(\cO_X)\to {\rm Gr}^F_{p-n}{\rm DR}_X\big(\cO_X(*Z)\big)$$
gets identified with a morphism
\begin{equation}\label{eq_ident_alpha}
\Omega_X^{n-p}[p]\to \R \mu_*\Omega_Y^{n-p}({\rm log}\,E)[p].
\end{equation}

Since ${\rm Gr}^F_{p-n}{\rm DR}_X(-)$ is an exact functor, we have an exact sequence of complexes
$$0\to {\rm Gr}^F_{p-n}{\rm DR}_X(\cO_X)\to {\rm Gr}^F_{p-n}{\rm DR}_X\big(\cO_X(*Z)\big)\to {\rm Gr}^F_{p-n}{\rm DR}_X(\cH_Z^1(\cO_X))\to 0,$$
which induces an exact triangle 
$$\R \cH om_{\cO_X}\big({\rm Gr}^F_{p-n}{\rm DR}_X(\cH_Z^1(\cO_X)),\omega_X\big)\longrightarrow \R \cH om_{\cO_X}\big({\rm Gr}^F_{p-n}{\rm DR}_X(\cO_X(*Z)),\omega_X\big)\overset{\beta}\longrightarrow$$
$$\overset{\beta}\longrightarrow\R \cH om_{\cO_X}\big({\rm Gr}^F_{p-n}{\rm DR}_X(\cO_X),\omega_X)\overset{+1}{\longrightarrow}.$$
As $\beta$ is the Grothendieck dual of $\alpha$, using (\ref{eq_ident_alpha}) and (\ref{eq2_remark2}) we can identify it with a morphism
\begin{equation}\label{identif_beta}
\R \mu_*\Omega^p_Y({\rm log}\,E)(-E)[-p]\longrightarrow \Omega_X^p[-p].
\end{equation}
We thus obtain the isomorphism in (\ref{eq_eq_DuBois}) from the exact triangle (\ref{eq_exact_trg}) if we show that the morphism (\ref{identif_beta}) is the same as the one in 
(\ref{eq_exact_trg}). Up to shift, the target of these morphisms is a locally free sheaf on a smooth variety, hence in order to show that they coincide it is enough to do so on an open subset
whose complement has codimension $\geq 2$. Since the definition of all morphisms we considered is compatible with restriction to open subsets of $X$, we can thus reduce to the case when
$\mu$ is an isomorphism so, in particular, $Z$ is smooth. In this case $\alpha$ is obtained by applying ${\rm Gr}^F_{p-n}{\rm DR}_X(-)$ to the following
morphism between the filtered resolutions of the right $\cD_X$-modules $\omega_X$ and $\omega_X(*Z)$:

\[\label{eq_diagram_alpha}
\xymatrix{
0 \ar[r] &\cD_X\ar[d]\ar[r] & \Omega_X^1\otimes_{\cO_X}\cD_X\ar[d]\ar[r] &\cdots\ar[r]  & \Omega_X^{n}\otimes_{\cO_X}\cD_X\ar[d]\ar[r] & 0 \\
0\ar[r] &\cD_X\ar[r] & \Omega_X^1({\rm log}\,E)\otimes_{\cO_X}\cD_X\ar[r] &\cdots\ar[r] & \Omega_X^n({\rm log}\,E)\otimes_{\cO_X}\cD_X\ar[r] & 0,
}
\]
in which the vertical morphisms are induced by the inclusions $\Omega_X^i\hookrightarrow\Omega_X^i({\rm log}\,E)$.
Using this, it is now straightforward to see that $\beta$ is equal to the morphism in the exact triangle (\ref{eq_exact_trg}).
\end{proof}

\begin{rmk}
Because of the compatibility between duality for mixed Hodge modules and duality for the corresponding de Rham
complexes in (\ref{eq_compatibility}), the isomorphism (\ref{eq_eq_DuBois}) is equivalent to an isomorphism
\begin{equation}\label{eq_DuBois2}
\underline{\Omega}_Z^p\simeq {\rm Gr}^F_{n-p}{\rm DR}_X\big({\mathbf D}(\cH^1_Z(\cO_X))\big)[p+1-n].
\end{equation}
The existence of a canonical such isomorphism was originally obtained by Saito, see \cite[Section 2]{Saito09}. 
While our arguments are more direct, they have the drawback that we only get an isomorphism in the derived category of $X$ (as opposed to the derived category of $Z$). Furthermore, this isomorphism is not uniquely determined, as it is associated to two different cones of a certain morphism. However, for our purpose, the existence of such an isomorphism will suffice.
\end{rmk}

\subsection{Steenbrink's vanishing theorem}\label{rmk_Steenbrink}
Since the de Rham complex of any filtered $\cD_X$-module is supported in nonpositive degrees, it follows from
(\ref{eq_DuBois2}) that 
$$\cH^q(\underline{\Omega}_Z^p)=0 \,\,\,\,\,\,{\rm  for~all} \,\,\,\,\,\, q\geq n-p.$$ 
(This is a special case of a
vanishing result that holds for arbitrary varieties $Z$; see \cite[Theorem~7.29]{PetersSteenbrink}.)
This in turn implies via the exact triangle (\ref{eq_exact_trg}) the fact that
$$R^q \mu_*\Omega_Y^p({\rm log}\,E)(-E)=0\quad\text{for}\quad q>n-p,$$
the assertion of Steenbrink's vanishing theorem in our setting; see \cite[Theorem~2]{Steenbrink}.

\section{Proof of the vanishing results}

We keep the notation from Section~\ref{section_notation}.
Before proving Theorem~\ref{thm1_intro} and related results, we make some preliminary considerations. 

We consider the pole order filtration $P_\bullet$ on $\cO_X(*Z)$ defined in Section \ref{scn:LHFME}. 
We also denote by $P_{\bullet}$ the induced filtration on 
$\cH^1_Z(\cO_X)=\cO_X(*Z)/\cO_X$. For every nonnegative integer 
$p$, with $p\leq n$, consider the complex
$$C_p^{\bullet}={\rm Gr}^P_{p-n}{\rm DR}_X\big(\cH_Z^1(\cO_X)\big).$$
In other words, $C_p^{\bullet}$ is the following complex, placed in cohomological degrees $-p,\ldots,0$:
\begin{equation}\label{eq_formula_C}
0\to \Omega_X^{n-p}\otimes_{\cO_X}\cO_Z(Z)\to \Omega_{X}^{n-p+1}\otimes_{\cO_X}\cO_Z(2Z)\to\cdots\to\omega_X\otimes_{\cO_X}\cO_Z\big((p+1)Z\big)
\to 0.
\end{equation}
If $U\subseteq X$ is an open subset where $Z\cap U$ is defined by $f\in\cO_X(U)$, then on $U$
the differential of the complex acts at $\Omega_X^{n-p+i}\otimes\cO_Z\big((i+1)Z)$ as 
\begin{equation}\label{eq_descr_C}
\eta\otimes [1/f^{i+1}] \,\,\,\, \mapsto \,\,\,\, -(i+1)(\eta\wedge df)\otimes [1/f^{i+2}].
\end{equation}
It will be convenient to also consider $C_{-1}^{\bullet}=0$. 

Since the Hodge filtration $F_{\bullet}$ on $\cO_X(*Z)$ satisfies $F_k\cO_X(*Z)\subseteq P_k \cO_X(*Z)$ for all $k$, we have a canonical morphism
\begin{equation}\label{phi}
\varphi_p\colon {\rm Gr}^F_{p-n}{\rm DR}_X\big(\cH^1_Z(\cO_X))\to C_p^{\bullet}.
\end{equation}

\smallskip

\begin{lem}\label{lem_translation1}
For every $p\geq 0$, if $\widetilde{\alpha}(Z)\geq p$, then the morphism of complexes $\varphi_p$ is injective and ${\rm Coker}(\varphi_p)$ is supported in
cohomological degree $0$. Moreover, we have ${\rm Coker}(\varphi_p)=0$ if and only if $\widetilde{\alpha}(Z)\geq p+1$. 
\end{lem}

\begin{proof}
As explained in Section \ref{scn:LHFME}, we have $\widetilde{\alpha}(Z)\geq p$ if and only if 
$F_k\cO_X(*Z)=P_k\cO_X(*Z)$ for all $k\leq p-1$, or equivalently, $F_k\cH^1_Z(\cO_X)=P_k\cH_Z^1(\cO_X)$ for all $k\leq p-1$. 
We thus see that if  we denote by $\varphi_p^i$ the component of $\varphi_p$ in cohomological degree $i$, then 
the hypothesis implies that  $\varphi_p^i$ is an isomorphism for all $i\neq 0$ and $\varphi_p^0$ is injective.
Moreover, ${\rm Coker}(\varphi_p^0)=0$ if and only if $\widetilde{\alpha}(Z)\geq p+1$. 
\end{proof}

\begin{cor}\label{cor_lem_translation1}
The following assertions are equivalent:
\begin{enumerate}
\item[i)] $\widetilde{\alpha}(Z)\geq p+1$.
\item[ii)] $\varphi_k$ is an isomorphism of complexes for all $0\leq k\leq p$.
\item[iii)] $\varphi_k$ is a quasi-isomorphism for all $0\leq k\leq p$. 
\end{enumerate}
\end{cor}

\begin{proof}
The implication i)$\Rightarrow$ii) follows from the lemma, while the implication ii)$\Rightarrow$iii) is trivial. 
The implication iii)$\Rightarrow$i) follows by induction on $p$, with the case $p=-1$ being trivial. 
The induction step follows from the lemma and the fact that an injective morphism of complexes is a quasi-isomorphism
if and only if its cokernel is acyclic. 
\end{proof}

We will use the following consequence:

\begin{prop}\label{prop_translation}
For every $p\ge 0$, if $\widetilde{\alpha}(Z)\geq p+1$, then we have an isomorphism
$$ \R \cH om_{\cO_X}\big(C_p^{\bullet},\omega_X[n]\big)\simeq \underline{\Omega}_Z^p[n-p-1].$$
\end{prop}

\begin{proof}
It follows from Corollary~\ref{cor_lem_translation1} that $\varphi_p$ is an isomorphism. Applying $\R \cH om_{\cO_X}\big(-,\omega_X[n]\big)$ and using the fact that 
by Lemma~\ref{eq_DuBois} we have an isomorphism
$$\R \cH om_{\cO_X}\big({\rm Gr}^F_{p-n}{\rm DR}_X(\cH_Z^1(\cO_X)),\omega_X[n]\big)\simeq \underline{\Omega}_Z^k[n-p-1],$$
we obtain the assertion in the proposition.
\end{proof}

For future reference, we recall that for every locally free sheaf $\cE$ on $X$, we have
\begin{equation}\label{eq_Ext}
{\mathcal Ext}_{\cO_X}^1(\cE\vert_Z,\omega_X)\simeq\cE^{\vee}\otimes_{\cO_X}\omega_Z \,\,\,\,\,\,
{\rm and} \,\,\,\,\,\, {\mathcal Ext}_{\cO_X}^j(\cE\vert_Z,\omega_X)=0 \,\,\,\,{\rm for} \,\,\,\,j \neq 1.
\end{equation}
We will need one more result about the complexes $C_p^{\bullet}$. 

\begin{lem}\label{vanishing_for_C}
For every $p\geq 0$, we have 
\begin{equation}\label{eq_pf_thm1}
{\mathcal Ext}_{\cO_X}^j(C_p^{\bullet},\omega_X)=0\quad\text{for all}\quad j> p+1.
\end{equation}
If, in addition, we assume that $\widetilde{\alpha}(Z)\geq p+1$, then we also have
\begin{equation}\label{eq5_pf_thm1}
{\mathcal Ext}_{\cO_X}^j(C_p^{\bullet},\omega_X)=0\quad\text{for all}\quad j<p+1
\end{equation}
and an exact sequence
\begin{equation}\label{eq2_pf_thm1}
0\to {\mathcal Ext}_{\cO_X}^{p}(C_{p-1}^{\bullet},\omega_X)\otimes\cO_X(-Z)\to \Omega_X^p\vert_Z\to {\mathcal Ext}_{\cO_X}^{p+1}(C_p^{\bullet},\omega_X)\to 0.
\end{equation}
\end{lem}

\begin{proof}
It follows from the description of the complex 
$C_p^{\bullet}$  and its differential in (\ref{eq_formula_C}) and (\ref{eq_descr_C}) 
that $C_{p-1}^{\bullet}\otimes_{\cO_X}\cO_X(Z)$ 
is isomorphic to the ``stupid" truncation 
$\sigma^{\geq -p+1}(C_p^{\bullet})$.
 Therefore we have a short exact sequence of complexes
\begin{equation}\label{eq1_lem_van1}
0\longrightarrow C_{p-1}^{\bullet}\otimes\cO_X(Z)\longrightarrow C_{p}^{\bullet}\longrightarrow \Omega_X^{n-p}\otimes\cO_Z(Z)[p]\longrightarrow 0.
\end{equation}
Using the vanishings in (\ref{eq_Ext}), we deduce from (\ref{eq1_lem_van1}) that for $j>p+1$, we have an exact sequence
$$
0={\mathcal Ext}_{\cO_X}^{j-p}\big(\Omega_X^{n-p}\otimes\cO_Z(Z),\omega_X\big)\to {\mathcal Ext}^j(C_p^{\bullet},\omega_X)\to 
{\mathcal Ext}^j_{\cO_X}(C_{p-1}^{\bullet},\omega_X)\otimes\cO_X(-Z).
$$
Using this, we obtain the vanishing in (\ref{eq_pf_thm1}) by induction on $p\geq -1$, the case $p=-1$ being trivial.

Next, suppose that $\widetilde{\alpha}(Z)\geq p+1$. Using Proposition~\ref{prop_translation}, we deduce that for $j<p+1$ we have
$${\mathcal Ext}^j(C_p^{\bullet},\omega_X)\simeq \cH^{j-p-1}(\underline{\Omega}_Z^p)=0,$$
where the vanishing follows from the fact that $\underline{\Omega}_Z^p$ has no nonzero cohomology sheaves in negative degrees; see \cite[Proposition~7.30b]{PetersSteenbrink}, or simply use the exact triangle (\ref{eq_exact_trg}).
This shows ($\ref{eq5_pf_thm1}$). Moreover, using this vanishing, from (\ref{eq1_lem_van1}) we get the short exact sequence
$$0\to {\mathcal Ext}_{\cO_X}^{p}(C_{p-1}^{\bullet},\omega_X)\otimes\cO_X(-Z) \to {\mathcal Ext}^1_{\cO_X}\big(\Omega_X^{n-p}\otimes\cO_Z(Z),\omega_X\big)\to {\mathcal Ext}_{\cO_X}^{p+1}(C_p^{\bullet},\omega_X)\to 0.$$
This shows ($\ref{eq2_pf_thm1}$) once we note that the isomorphism in (\ref{eq_Ext}) gives 
$${\mathcal Ext}^1_{\cO_X}\big(\Omega_X^{n-p}\otimes\cO_Z(Z),\omega_X\big)
\simeq (\Omega_X^{n-p})^{\vee}\otimes\cO_Z(-Z)\otimes\omega_Z\simeq\Omega_X^p\vert_Z.$$
This completes the proof of the lemma.
\end{proof}

We can now prove the first result stated in the Introduction. 

\begin{proof}[Proof of Theorem~\ref{thm1_intro}]
If $Z$ is smooth the statement is trivial, hence from now on we assume that $Z$ is singular.
We fix a nonnegative integer $p$ such that $\widetilde{\alpha}(Z)\geq p+1$. The fact that the canonical morphism
$\Omega_Z^p\to \underline{\Omega}_Z^p$ is an isomorphism is equivalent to having $\cH^i(\underline{\Omega}^p_Z)=0$ for $i\neq 0$ and the canonical morphism
\begin{equation}\label{eq_morphism_H0}
\Omega_Z^p\to \cH^0(\underline{\Omega}_Z^p)
\end{equation}
being an isomorphism. By combining Proposition~\ref{prop_translation} and Lemma~\ref{vanishing_for_C}, we see that since $\widetilde{\alpha}(Z)\geq p+1$, 
we have
$$\cH^i(\underline{\Omega}^p_Z)\simeq {\mathcal Ext}_{\cO_X}^{p+1+i}(C_p^{\bullet},\omega_X)=0\quad\text{for}\quad i\neq 0.$$
Therefore we are left with showing that (\ref{eq_morphism_H0}) is an isomorphism. 

Note that when $p=0$ the assertion in the theorem follows from \cite[Theorem~0.5]{Saito09}.
We thus may and will assume that $p\geq 1$. 
In this case, since $\widetilde{\alpha}(Z)\geq 2$, we deduce from \cite[Theorem~0.4]{Saito93} that $Z$ has rational singularities, hence it is normal. 

Consequently, since
the restriction of (\ref{eq_morphism_H0}) to the smooth locus of $X$ is an isomorphism, it is enough to prove that both sheaves $\Omega_Z^p$ and $\cH^0(\underline{\Omega}_Z^p)$
satisfy Serre's property $S_2$: this implies that if $j\colon U=Z\smallsetminus Z_{\rm sing} \hookrightarrow Z$ is the inclusion of the smooth locus of $Z$, then (\ref{eq_morphism_H0}) gets identified
with 
$$j_*(\Omega_Z^p\vert_U)\to j_*\big(\cH^0(\underline{\Omega}_Z^p)\vert_U\big).$$
In fact, we will prove the following stronger fact: for every point $x\in Z_{\rm sing}$, not necessarily closed, we have
\begin{equation}\label{eq_formula_pd}
{\rm depth}(\Omega_{Z,x}^p)\geq \dim(\cO_{X,x})-p-1\quad\text{and}\quad {\rm depth}\big(\cH^0(\underline{\Omega}_Z^p)_x\big)\geq \dim(\cO_{X,x})-p-1.
\end{equation}
Note that since $Z$ is singular and $\widetilde{\alpha}(Z)\geq p+1$, we have ${\rm codim}_X(Z_{\rm sing})\geq 2p+2$ by Lemma~\ref{lem_singular_locus}.
For every $x\in Z_{\rm sing}$ as above, we thus have $\dim(\cO_{X,x})-p-1\geq p+1$. As the restrictions of $\Omega_Z^p$ and $\cH^0(\underline{\Omega}_Z^p)$
to $U$ are locally free, we thus deduce from (\ref{eq_formula_pd}) that if $Z$ is singular, then both $\Omega_Z^p$ and $\cH^0(\underline{\Omega}_Z^p)$
satisfy Serre's property $S_{p+1}$, hence also property $S_2$ (recall that $p\geq 1$).

We first treat $\cH^0(\underline{\Omega}_Z^p)$. Note that since $\widetilde{\alpha}(Z)\geq p+1$, it follows from Proposition~\ref{prop_translation} that
we have an isomorphism
\begin{equation}\label{eq_H0}
\cH^0(\underline{\Omega}_Z^p)\simeq {\mathcal Ext}_{\cO_X}^{p+1}(C_p^{\bullet},\omega_X).
\end{equation}
Moreover, for every $k$ with $0\leq k\leq p$, 
we can use the exact sequence in Lemma~\ref{vanishing_for_C} to deduce
$${\rm depth}\big({\mathcal Ext}_{\cO_X}^{k+1}(C_k^{\bullet},\omega_X)_x\big)\geq \min\big\{{\rm depth}((\Omega_X^k\vert_Z)_x), {\rm depth}\big({\mathcal Ext}_{\cO_X}^{k}(C_{k-1}^{\bullet},\omega_X)_x\big)-1\big\}$$
$$=\min\big\{\dim(\cO_{X,x})-1, {\rm depth}\big({\mathcal Ext}_{\cO_X}^{k}(C_{k-1}^{\bullet},\omega_X)_x\big)-1\big\},$$
where the inequality follows for example from \cite[Proposition~1.2.9]{BH} and the equality follows from the fact that $\Omega_X^k\vert_Z$ is a locally free $\cO_Z$-module.
Using induction on $k$, with $0\leq k\leq p$, and the fact that $C_{-1}^{\bullet}=0$, we conclude that
$${\rm depth}\big({\mathcal Ext}_{\cO_X}^{k+1}(C_k^{\bullet},\omega_X)_x\big)\geq \dim(\cO_{X,x})-k-1.$$
In particular, due to the isomorphism (\ref{eq_H0}), for $k=p$ we get the second inequality in (\ref{eq_formula_pd}) . 

In order to prove the first inequality in (\ref{eq_formula_pd}), we may argue locally, and thus assume that 
$Z$ is defined in $X$ by some $f\in\cO_X(X)$. In this case we have the presentation
$$\cO_Z\overset{df}\longrightarrow \Omega_X^1\vert_Z\longrightarrow \Omega_Z^1\longrightarrow 0$$
and the zero-locus of $df$ is $Z_{\rm sing}$, whose codimension in $Z$ is $\geq 2p+1\geq p+2$. 
Since $Z$ is Cohen-Macaulay and we have in particular ${\rm codim}_Z(Z_{\rm sing})\geq p$, 
it follows from the description of depth via Koszul homology, see \cite[Theorem~16.8]{Matsumura}, 
that the following complex
$$0\longrightarrow\cO_Z\overset{df}\longrightarrow \Omega_X^1\vert_Z\overset{-\wedge df}\longrightarrow \cdots\overset{-\wedge df}\longrightarrow
\Omega_X^p\vert_Z\longrightarrow \Omega_Z^p\to 0$$
is exact (note that exactness at the last two terms holds in general).
Breaking this into short exact sequences, localizing at $x$,  and using again \cite[Proposition~1.2.9]{BH} and the fact that the sheaves $\Omega_X^k\vert_Z$ are 
locally free sheaves of $\cO_Z$-modules, proceeding as in the previous paragraph we obtain the first inequality in (\ref{eq_formula_pd}).
%$${\rm depth}(\Omega_{Z,x}^p) \geq \dim(\cO_{X,x})-p-1.$$
\end{proof}

\begin{rmk}\label{rmk_reflexive}
We note that in the proof of Theorem~\ref{thm1_intro} we have shown that if $Z$ is a singular hypersurface such that $\widetilde{\alpha}(Z)\geq p+1$, for some $p\geq 1$, then
$\Omega_Z^p$ satisfies Serre's condition $S_{p+1}$, and thus the condition $S_2$. (If $Z$ is smooth, then of course all sheaves $\Omega_Z^k$ are Cohen-Macaulay). In particular, since $Z$ is normal,
it follows that
$\Omega_Z^p$ is a reflexive sheaf; see \cite[Proposition~1.4.1]{BH}.
\end{rmk}

\begin{rmk}\label{rational-H0}
Under the assumptions of Theorem~\ref{thm1_intro}, if $p\geq 1$, then one can in fact describe the sheaves 
$\cH^0(\underline{\Omega}_Z^q)$ concretely, as the reflexive hull of $\Omega_Z^q$, for \emph{all} $q$ with $0\leq q\leq n$. First, with no assumptions on $Z$, they can be identified with the sheaves of $h$-differentials, namely 
$$\cH^0(\underline{\Omega}_Z^q)\simeq {\Omega^q_{\rm h}}|_{Z}$$
for each $q$; this follows from \cite[Theorem~7.12]{Huber} (see also the notation after Remark 6.13 in \emph{loc.}\,\emph{cit.}). On the other hand, as noted in Section \ref{scn:LHFME} (the second of the numerical criteria for minimal exponents), by 
\cite[Theorem~0.4]{Saito93} the condition 
$\widetilde{\alpha}(Z)\geq 2$ implies that $Z$ has rational singularities (hence in particular it is normal). Since $Z$ is a hypersurface, this is equivalent to $Z$ having klt singularities
by \cite[Corollary~11.13]{Kollar}. In this case, if $j\colon Z_{\rm sm}\hookrightarrow Z$ is the inclusion of the smooth locus, then 
$$ {\Omega^q_{\rm h}}|_{Z} \simeq j_*\Omega_{Z_{\rm sm}}^q \simeq (\Omega_Z^q)^{\vee \vee}$$ 
for all $q$, by \cite[Theorem~5.4]{Huber}. More recently, this was shown to hold for all varieties with rational singularities by Kebekus-Schnell \cite[Corollary~1.11]{KeS}.
\end{rmk}

While the statement and argument for the vanishing in Theorem \ref{thm1_intro} are particularly transparent, a stronger statement can be made about the vanishing of individual cohomologies, in terms of the size of the loci in $Z$ where the minimal exponent is small.

\begin{thm}\label{stronger-vanishing}
Let $p$ be a nonnegative integer. If the locus $Z_p$ of points $x\in Z$ with $\widetilde{\alpha}_x(Z)<p+1$ ${\rm (}$equivalently, the closed subset defined by the Hodge ideal 
$I_p (Z)$${\rm )}$ satisfies ${\rm codim}_X Z_p> i+p+2$ for some $i\geq 1$, then $\cH^i(\underline{\Omega}_Z^p)=0$. 
\end{thm}
\begin{proof}
%We use freely the notation and arguments from the proof of Theorem \ref{thm1_intro}.
Note first that by Corollary~\ref{cor_lem_translation1} the morphism 
$$\varphi_p\colon {\rm Gr}_{p-n}{\rm DR}_X\big(\cH^1_Z(\cO_X), F\big)\to C_p^{\bullet}$$
as in ($\ref{phi}$) has the property that both $A^{\bullet}={\rm ker}(\varphi_p)$ and $B^{\bullet}={\rm coker}(\varphi_p)$ have all terms supported on $Z_p$. 
Moreover, these complexes are concentrated in cohomological degrees $\leq 0$. 
%We have seen in the proof of Theorem~\ref{thm1_intro}
The first assertion in Lemma~\ref{vanishing_for_C} gives ${\mathcal Ext}^{i+p+1}_{\cO_X}(C_p^{\bullet},\omega_X)=0$. The short exact sequences
$$0={\mathcal Ext}_{\cO_X}^{i+p+1}(C_p^{\bullet},\omega_X)\to {\mathcal Ext}^{i+p+1}_{\cO_X}\big({\rm im}(\varphi_p),\omega_X\big)
\to {\mathcal Ext}^{i+p+2}_{\cO_X}(B^{\bullet},\omega_X)$$
and
$${\mathcal Ext}_{\cO_X}^{i+p+1}({\rm im}(\varphi_p),\omega_X)\to {\mathcal Ext}^{i+p+1}_{\cO_X}\big({\rm Gr}_{p-n}{\rm DR}_X(\cH_Z^1(\cO_X)),\omega_X\big)
\simeq \cH^i(\underline{\Omega}_Z^p)
\to {\mathcal Ext}^{i+p+1}_{\cO_X}(A^{\bullet},\omega_X),
$$
where the middle isomorphism in the latter follows from Lemma \ref{eq_DuBois}, imply that in order to conclude it is enough to show that
$${\mathcal Ext}_{\cO_X}^{i+p+1}(A^{\bullet},\omega_X)={\mathcal Ext}_{\cO_X}^{i+p+2}(B^{\bullet},\omega_X)=0.$$

It is thus suffices to check that if $\cF^{\bullet}$ is a complex on $X$ concentrated in degrees $\leq 0$, and
${\rm codim}_X{\rm Supp}(\cF^q) >m$ for all $q$, then ${\mathcal Ext}_{\cO_X}^m(\cF^{\bullet},\omega_X)=0$.
This follows from the hypercohomology spectral sequence
$$E_1^{i,j}={\mathcal Ext}_{\cO_X}^j(\cF^{-i},\omega_X)\Rightarrow {\mathcal Ext}_{\cO_X}^{i+j}(\cF^{\bullet},\omega_X),$$
since when $i+j=m$, we have $E_1^{i,j}=0$: indeed, we may assume that $i\geq 0$, hence $j=m-i\leq m$
and then ${\mathcal Ext}_{\cO_X}^j(\cF^{-i},\omega_X)=0$ as ${\rm codim}_X{\rm Supp}(\cF^{-i}) >m\geq j$.
%This implies that ${\mathcal Ext}_{\cO_X}^m(\cF^{\bullet},\omega_X)=0$.
\end{proof}

An immediate consequence is a range of automatic vanishing in terms of the size of the singular locus of $Z$.

\begin{cor}\label{cor:singular}
If the singular locus of the hypersurface $Z$ has dimension $s$, then for all $p \ge 0$ we have
$$\cH^i(\underline{\Omega}_Z^p)=0 \,\,\,\,\,\,{\rm for} \,\,\,\,\,\,1\leq i < n - s - p -2.$$
\end{cor}

\begin{rmk}\label{non-DB}
The two results above are of course relevant even in the non-Du Bois case, or equivalently when we have $\widetilde{\alpha}_x(Z)< 1$ at some points, as Theorem \ref{stronger-vanishing} implies that $\cH^i(\underline{\Omega}_Z^0)=0 $ for $i < n - s -2$ if $s$ is the dimension of the non-Du Bois locus. For instance, if $Z$ has isolated singularities, then $\cH^i (\underline{\Omega}^0_Z) \neq 0$ can happen only for $i = 0$ and $i = n-2$, and it does happen for both if $Z$ is not Du Bois.\footnote{Note that $\cH^{n-1}(\underline{\Omega}_Z^0)=0$ by Theorem~\ref{thm3_intro}.}
\end{rmk}

\medskip

We next deduce the two corollaries of Theorem~\ref{thm1_intro}.

\begin{proof}[Proof of Corollary~\ref{cor1_intro}]
Since the morphism $\Omega_X^k\to\Omega_Z^k$ is surjective for every $k\geq 0$, the assertion follows directly from Theorem~\ref{thm1_intro} via the exact triangle (\ref{eq_exact_trg}). It is worth noting that Corollary~\ref{cor1_intro} is in fact equivalent to the vanishings $\cH^i (\underline{\Omega}_Z^p) = 0$ for $i >0$, plus the surjectivity of the natural map 
$\Omega^p_X \to \cH^0 (\underline{\Omega}_Z^p)$.
\end{proof}

\begin{proof}[Proof of Corollary~\ref{cor0_intro}]
The assertion follows directly from Theorem~\ref{thm1_intro} and the version of the Akizuki-Nakano vanishing theorem
for the graded pieces of the Du Bois complex:
$$H^i(Z,\underline{\Omega}_Z^{j}\otimes L)=0\quad\text{for}\quad i+j>\dim Z,$$
see \cite[Theorem V.5.1]{GNPP}.
%\cite[Theorem~7.29]{PetersSteenbrink}. 
\end{proof}

Using a similar approach to that in the proof of Theorem~\ref{thm1_intro}, we obtain the vanishing result in Theorem~\ref{thm3_intro}, as follows.

\begin{proof}[Proof of Theorem~\ref{thm3_intro}]
We may and will assume that $Z$ is singular, in which case our hypothesis implies $q\leq n/2$; in particular, we have $q\neq n$.
It is enough to prove that $\cH^{n-q-1}(\underline{\Omega}_Z^q)=0$: indeed, 
the vanishing of $R^{n-q} \mu_*\Omega_X^q({\rm log}\,E)(-E)$ then follows from
the exact triangle (\ref{eq_exact_trg}) since $q\neq n$ gives $\cH^{n-q}(\Omega_X^q)=0$. 

Furthermore, we may assume that $q\leq n-2$: if $q=n-1$, the fact that $q\leq n/2$ implies $n=2$. Moreover, the hypothesis that
$\widetilde{\alpha}(Z)\geq q=1$ gives that $(X,Z)$ is log canonical, and it is well known that this can only happen for nodal curves (note that in this case we
clearly have $\cH^{0}(\underline{\Omega}_Z^1)\neq 0$).

Using Lemma \ref{eq_DuBois}, we see that
$$\cH^{n-q-1}(\underline{\Omega}_Z^q)
\simeq {\mathcal Ext}^n_{\cO_X}\big({\rm Gr}_{q-n}^F{\rm DR}_X(\cH^1_Z(\cO_X)),\omega_X\big).$$

Note now that since $\widetilde{\alpha}(Z)\geq q$, it follows from Lemma~\ref{lem_translation1} that the morphism of complexes
$$\varphi_q\colon {\rm Gr}_{q-n}^F{\rm DR}_X(\cH^1_Z(\cO_X))\to C_q^{\bullet}$$
is injective and its cokernel is a sheaf $\cF$ (supported in cohomological degree $0$).
We thus have an exact sequence
$${\mathcal Ext}^n_{\cO_X}(C_q^{\bullet},\omega_X)\to {\mathcal Ext}^n_{\cO_X}\big({\rm Gr}_{q-n}^F{\rm DR}_X(\cH^1_Z(\cO_X)),\omega_X\big)
\to {\mathcal Ext}^{n+1}_{\cO_X}(\cF,\omega_X)=0,$$
hence it is enough to show that ${\mathcal Ext}^n_{\cO_X}(C_q^{\bullet},\omega_X)=0$. 

We use the short exact sequence of complexes ($\ref{eq1_lem_van1}$) as in the proof of Lemma~\ref{vanishing_for_C}:
$$0\to C_{q-1}^{\bullet}\otimes\cO_X(Z) \to C_q^{\bullet}\to \Omega_X^{n-q}\otimes\cO_Z(Z)[q]\to 0.$$
We deduce the existence of an exact sequence
$${\mathcal Ext}_{\cO_X}^{n-q}(\Omega_X^{n-q}\vert_Z,\omega_X)\otimes\cO_X(-Z)\to {\mathcal Ext}^n_{\cO_X}(C_q^{\bullet},\omega_X) \to 
{\mathcal Ext}_{\cO_X}^n(C^{\bullet}_{q-1},\omega_X)\otimes\cO_X(-Z).$$
The first term vanishes by (\ref{eq_Ext}), since $q\neq n-1$, and the third term vanishes by Lemma~\ref{vanishing_for_C},
since $q\neq n$. We thus have ${\mathcal Ext}^n_{\cO_X}(C_q^{\bullet},\omega_X)=0$,
completing the proof of the theorem.
\end{proof}

\section{Proof of the non-vanishing result}\label{last_section}

The proof of Theorem~\ref{thm4_intro} makes use of the $V$-filtration, so we begin with a very brief review of this notion. 
For more details, we refer for example to \cite[Section~3.1]{Saito-MHP} or \cite[Section~2]{MP2}. We keep the assumptions from Section~\ref{section_notation}, but we assume in addition that $Z$ is defined by $f\in\cO_X(X)$.

It is common to denote by $B_f$ the $\cD$-module push-forward $\iota_+\cO_X$, where $\iota\colon X\hookrightarrow W=X\times {\mathbf A}^1$ is the graph embedding
$\iota(x)=\big(x,f(x)\big)$. If $t$ denotes the coordinate on ${\mathbf A}^1$, then there is an isomorphism 
$$B_f \simeq \cO_X[t]_{f - t} / \cO_X[t] \simeq \bigoplus_{i\geq 0}\cO_X\cdot\partial_t^i\delta,$$
where $\delta$ denotes the class of $\frac{1}{f-t}$, and the actions of $t$ and of a derivation $P\in {\rm Der}_{{\mathbf C}}(\cO_X)$ are given by
$$t\cdot h\partial_t^i\delta=fh\partial_t^i\delta-ih\partial_t^{i-1}\delta\quad\text{and}\quad P\cdot h\partial_t^i\delta=
P(h)\partial_t^i\delta-P(f)h\partial_t^{i+1}\delta.$$
The $\cD_W$-module $B_f$ carries a (Hodge) filtration given by 
$$F_{p+1}B_f=\bigoplus_{0\leq i\leq p}\cO_X\cdot\partial_t^i\delta.$$
This filtered $\cD_X$-module underlies a pure Hodge module of weight $n$. 

When dealing with duality, it is more common to use right $\cD_X$-modules. In order to avoid confusion when citing various results,
we will follow this tradition. Recall that we have an equivalence of categories between left and right (filtered) $\cD$-modules; see Section~\ref{subsection_DR_duality}. 
For example, the right $\cD$-module
corresponding to $B_f$ is 
$$B_f^r:=\iota_+\omega_X=\omega_X\otimes_{\cO_X}B_f.$$

The $V$-filtration on $B_f$ is a decreasing, exhaustive, discrete, and left continuous filtration
$(V^\alpha B_f)_{\alpha\in\Q}$ parametrized by rational numbers.
It is characterized uniquely by a number of properties listed for instance in \cite[Section~3.1]{Saito-MHP}.
The Hodge filtration on $B_f$ induces a filtration on each $V^{\alpha}B_f$ and thus
on ${\rm Gr}_V^{\alpha}B_f=V^{\alpha}B_f/V^{>\alpha}B_f$ as well. We have a corresponding $V$-filtration on $B_f^r$ given by
$V^{\alpha}B_f^r=\omega_X\otimes_{\cO_X}V^{\alpha}B_f$. Note that since the Hodge filtrations on ${\rm Gr}_V^{\alpha}B_f$
and ${\rm Gr}_V^{\alpha}B_f^r$ are induced by those on $B_f$ and $B_f^r$, respectively, these satisfy
\begin{equation}\label{eq_compat_B}
F_{p-n-1}{\rm Gr}_V^{\alpha}B_f^r=\omega_X\otimes_{\cO_X}F_p{\rm Gr}_V^{\alpha}B_f.
\end{equation}

An important fact is a result of Saito, see  \cite[(1.3.8)]{Saito-MLCT}, describing the minimal exponent via the $V$-filtration: if $q$ is a non-negative integer and $\alpha\in (0,1]$ is a rational number, then 
\begin{equation}\label{alpha-Vfil}
\widetilde{\alpha}(Z)\geq q+\alpha \iff \partial_t^q\delta\in V^{\alpha}B_f.
\end{equation}

This setting is relevant for us since the filtered right $\cD_X$-module $\cH^1_Z(\omega_X)$, corresponding to the 
$\cD_X$-module appearing in Lemma \ref{eq_DuBois}, is isomorphic to the cokernel of the morphism 
of filtered $\cD_X$-modules
$${\rm Gr}_V^0B_f^r\overset{\cdot t}\longrightarrow {\rm Gr}_V^1B_f^r$$
between the vanishing cycles and (a Tate twist of) the nearby cycles of $f$; see \cite[Section~2.24]{Saito-MHM}. 
It follows that ${\mathbf D}\big(\cH^1_Z(\omega_X)\big)$ is isomorphic to the kernel of the dual morphism
\begin{equation}\label{eq1_nearby}
{\mathbf D}({\rm Gr}_V^1B_f^r)\to {\mathbf D}({\rm Gr}_V^0B_f^r), 
\end{equation}
where ${\mathbf D}$ is the duality functor on filtered $\cD$-modules; see \cite[Section~2.4]{Saito-MHP}.
Since $B_f^r$ underlies a pure polarizable Hodge module of weight $n$, we have an isomorphism
${\mathbf D}(B_f^r)\simeq B_f^r(n)$. Here, for a filtered $\cD$-module $(\cM,F)$, we use the notation $(\cM,F)(q)$ for the filtered 
$\cD$-module $(\cM,F[q])$, where $F[q]_i\cM=F_{i-q}\cM$. 
Using the compatibility between duality and vanishing/nearby cycles proved by Saito in \cite[Theorem~1.6]{Saito_duality}, we also have isomorphisms of filtered (right) $\cD_X$-modules
$${\mathbf D}({\rm Gr}_V^1B_f^r)\simeq {\rm Gr}_V^1B^r_f(n+1)\quad\text{and}\quad {\mathbf D}({\rm Gr}_V^0B_f^r)\simeq {\rm Gr}_V^0B^r_f(n).$$
Moreover, the morphism 
(\ref{eq1_nearby}) gets identified (see \emph{loc.}\,\emph{cit.}) with the morphism 
\begin{equation}\label{eq2_nearby}
{\rm Gr}_V^1B^r_f(n+1)\overset{\cdot(-\partial_t)}\longrightarrow {\rm Gr}_V^0B_f^r(n).
\end{equation}

After this preparation, we can prove the result stated in the Introduction.

\begin{proof}[Proof of Theorem~\ref{thm4_intro}]
It follows from the formula (\ref{eq_DuBois2}) for the graded pieces of the Du Bois complex that for every $i$, we have 
$$\cH^i(\underline{\Omega}_Z^{n-p})\simeq \cH^{i-p+1}\big({\rm Gr}_p^F{\rm DR}_X{\mathbf D}(\cH_Z^1(\omega_X))\big).$$
On the other hand, it follows from the previous discussion that ${\rm Gr}_p^F{\rm DR}_X{\mathbf D}\big(\cH_Z^1(\omega_X)\big)$
is isomorphic to the kernel of the morphism 
$${\rm Gr}_{p-n-1}^F{\rm DR}_X\big({\rm Gr}_V^1(B_f^r)\big)\to {\rm Gr}_{p-n}^F{\rm DR}_X\big({\rm Gr}_V^0(B_f^r)\big)$$
induced by right multiplication with $\partial_t$. 
If we write these complexes explicitly in terms of left $\cD$-modules, using the identification in (\ref{eq_compat_B}), 
we see that ${\rm Gr}_p^F{\rm DR}_X{\mathbf D}\big(\cH_Z^1(\omega_X)\big)$
is the kernel of the morphism of complexes
\[\label{eq_diagram}
\xymatrix{
0 \ar[r] &\Omega_X^{n-p+1}\otimes {\rm Gr}^F_1{\rm Gr}^1_VB_f\ar[d]\ar[r] &\cdots\ar[r]  & \Omega_X^{n-1}\otimes {\rm Gr}_{p-1}^F{\rm Gr}_V^1B_f\ar[r]\ar[d] & \omega_X\otimes {\rm Gr}_p^F{\rm Gr}_V^{1}B_f\ar[r]\ar[d] & 0 \\
0\ar[r] &\Omega_X^{n-p+1}\otimes {\rm Gr}^F_2{\rm Gr}^0_VB_f\ar[r] &\cdots\ar[r]  & \Omega_X^{n-1}\otimes {\rm Gr}_{p}^F{\rm Gr}_V^0B_f\ar[r] & \omega_X\otimes {\rm Gr}_{p+1}^F
{\rm Gr}_V^{0}B_f\ar[r] & 0 
}
\]
placed in cohomological degrees $-(p-1),\ldots,0$ and in which the vertical maps are given by left multiplication by $\partial_t$. Under the assumption of the theorem, we will identify the top complex and show that in the bottom complex all terms are 0.

By ($\ref{alpha-Vfil}$), the condition $\widetilde{\alpha}(Z)>p$ is equivalent to the fact that $\partial_t^p\delta\in V^{>0}B_f$, and in fact $\partial_t^j\delta\in V^{>0}B_f$ for all $j\leq p$. We thus see that $F_{p+1}B_f\subseteq V^{>0}B_f$, hence
${\rm Gr}_j^F{\rm Gr}^0_VB_f=0$ for all $j\leq p+1$. Therefore the bottom complex in the above diagram is 0 and we conclude that
$$\cH^i(\underline{\Omega}_Z^{n-p})\simeq\cH^{i-p+1}\big({\rm Gr}_{p-n-1}^F{\rm DR}_X\big({\rm Gr}_V^1(B_f^r)\big)\big).$$

Again using ($\ref{alpha-Vfil}$), since $\widetilde{\alpha}(Z)\geq p$ we deduce that $\partial_t^j\delta\in V^1B_f$ for $j\leq p-1$. We conclude that for $1\leq j\leq p$, we have
$F_jV^1 B_f=\bigoplus_{i\leq j-1}\cO_X\cdot\partial_t^i\delta$. Note also that 
$F_jV^{>1}B_f=t\cdot F_jV^{>0}B_f$; 
this is a general property of filtered $\cD$-modules underlying mixed Hodge modules, see \cite[(3.2.1.2)]{Saito-MHP}.  
This implies that for $j\leq p$, we have
$$F_jV^{>1}B_f+F_{j-1}V^1B_f=\cO_X\cdot\delta\oplus\cdots\oplus \cO_X\cdot \partial_t^{j-2}\delta\oplus (f)\cdot \partial_t^{j-1}\delta.$$
We thus conclude that the morphism
$$\cO_X/(f)\to {\rm Gr}^F_j{\rm Gr}_V^1B_f=F_jV^1B_f/(F_jV^{>1}B_f+F_{j-1}V^1B_f)$$
that maps the class of $h$ to the class of $h\partial_t^{j-1}\delta$, is an isomorphism. 

Suppose now that we have algebraic local coordinates $x_1, \ldots, x_n$ in a neighborhood of $x$. A straightforward computation then shows that
$\cH^i(\underline{\Omega}_Z^{n-p})$ is the cohomology in degree $i-p+1$ of the ``stupid" truncation $\sigma^{\geq -p+1}$ of the Koszul complex on $\cO_X/(f)$
associated to the sequence $\partial f/\partial x_1,\ldots,\partial f/\partial x_n$. This immediately gives the formula for $\cH^{p-1}(\underline{\Omega}_Z^{n-p})$
in i).

Suppose now that $p\geq 3$ and $f$ has an isolated singularity at $x$. In this case, by Generic Smoothness, around $x$ the zero-locus of $J_f$
is contained in the hypersurface defined by $f$, hence it is equal to $\{x\}$. Therefore
the elements $\partial f/\partial x_1,\ldots,\partial f/\partial x_n$ form a regular sequence
in $\cO_{X,x}$, so that
$$\cH^i(\underline{\Omega}_Z^{n-p})_x\simeq {\rm Tor}_{p-1-i}^{\cO_{X,x}}(\cO_{X,x}/J_f, \cO_{Z,x})$$
for $1\leq i\leq p-1$.
The assertions in ii) are immediate consequences. (The vanishing statement also follows from Corollary~\ref{cor:singular},
and holds for an arbitrary isolated singularity.)
\end{proof}

\end{document}